\documentclass[letterpaper,11pt,oneside,reqno]{amsart}

\usepackage{bm}
\usepackage{caption}
\usepackage{afterpage}
\usepackage{enumitem}
\usepackage{bmpsize}
\usepackage{hyperref}

\usepackage{amsmath,amssymb,amsthm,amsfonts, dsfont}
\usepackage{relsize}
\usepackage{hyperref}
\usepackage{graphicx,color,subcaption}
\usepackage{upgreek}
\usepackage[mathscr]{euscript}

\allowdisplaybreaks
\numberwithin{equation}{section}

\usepackage{tikz}
\usetikzlibrary{shapes,arrows,positioning,decorations.markings}

\usepackage{float}
\usepackage{multicol}
\usepackage{array}
\usepackage{adjustbox}
\usepackage{cleveref}
\usepackage[DIV=12]{typearea}

\newtheorem{proposition}{Proposition}[section]
\newtheorem{lemma}[proposition]{Lemma}

\newtheorem{theorem}[proposition]{Theorem}

\theoremstyle{definition}

\newtheorem{remark}[proposition]{Remark}
\newtheoremstyle{qqq}
  {}   
  {}   
  {\slshape}  
  {0pt}       
  {\bfseries} 
  {.}         
  {5pt plus 1pt minus 1pt} 
  {}          
\theoremstyle{qqq}
\newtheorem{question}[proposition]{Open problem}

\makeatletter
\newcommand*{\da@rightarrow}{\mathchar"0\hexnumber@\symAMSa 4B }
\newcommand*{\da@leftarrow}{\mathchar"0\hexnumber@\symAMSa 4C }
\newcommand*{\xdashrightarrow}[2][]{%
  \mathrel{%
    \mathpalette{\da@xarrow{#1}{#2}{}\da@rightarrow{\,}{}}{}%
  }%
}
\newcommand{\xdashleftarrow}[2][]{%
  \mathrel{%
    \mathpalette{\da@xarrow{#1}{#2}\da@leftarrow{}{}{\,}}{}%
  }%
}
\newcommand*{\da@xarrow}[7]{%
  \sbox0{$\ifx#7\scriptstyle\scriptscriptstyle\else\scriptstyle\fi#5#1#6\m@th$}%
  \sbox2{$\ifx#7\scriptstyle\scriptscriptstyle\else\scriptstyle\fi#5#2#6\m@th$}%
  \sbox4{$#7\dabar@\m@th$}%
  \dimen@=\wd0 %
  \ifdim\wd2 >\dimen@
    \dimen@=\wd2 %
  \fi
  \count@=2 %
  \def\da@bars{\dabar@\dabar@}%
  \@whiledim\count@\wd4<\dimen@\do{%
    \advance\count@\@ne
    \expandafter\def\expandafter\da@bars\expandafter{%
      \da@bars
      \dabar@ 
    }%
  }%
  \mathrel{#3}%
  \mathrel{%
    \mathop{\da@bars}\limits
    \ifx\\#1\\%
    \else
      _{\copy0}%
    \fi
    \ifx\\#2\\%
    \else
      ^{\copy2}%
    \fi
  }%
  \mathrel{#4}%
}
\makeatother

\begin{document}
\title[Symmetry in perturbed GUE corners]{Parameter symmetry in perturbed GUE corners
	process and reflected 
	drifted Brownian motions}

\author[L. Petrov]{Leonid Petrov}
\address{L. Petrov, University of Virginia, Department of Mathematics,
141 Cabell Drive, Kerchof Hall,
P.O. Box 400137,
Charlottesville, VA 22904, USA,
and Institute for Information Transmission Problems,
Bolshoy Karetny per. 19, Moscow, 127994, Russia}
\email{lenia.petrov@gmail.com}

\author{Mikhail Tikhonov}
\address{M. Tikhonov, 
Lomonosov Moscow State University, Faculty of Physics,
Leninskie Gory, 1-2,
119991, Moscow, Russia,
and Institute for Information Transmission Problems,
Bolshoy Karetny per. 19, Moscow, 127994, Russia}
\email{me@mtikhonov.com}

\begin{abstract}
	The perturbed GUE corners ensemble
	is the joint distribution of 
	eigenvalues of all principal submatrices of 
	a matrix $G+\mathrm{diag}(\mathbf{a})$,
	where $G$ is the random matrix from the Gaussian Unitary Ensemble
	(GUE), and $\mathrm{diag}(\mathbf{a})$ is a fixed diagonal
	matrix.
	We introduce Markov transitions based on exponential jumps of 
	eigenvalues, and show that their successive application
	is equivalent in distribution to a deterministic shift of the 
	matrix. This result also leads to a new distributional symmetry
	for a family of reflected Brownian motions with 
	drifts coming from an arithmetic progression.
	
	The construction we present may be viewed as a 
	random matrix analogue of the 
	recent results of the first author and Axel Saenz \cite{PetrovSaenz2019backTASEP}.
\end{abstract}

\date{}

\maketitle

\setcounter{tocdepth}{3}

\section{Introduction}

\subsection{Couplings for perturbed GUE corners process}

The Gaussian Unitary Ensemble (GUE) is the most 
well-known random matrix model
\cite{mehta2004random}, 
\cite{Forrester-LogGas}, \cite{AndersonGuionnetZeitouniBook}.
This paper presents a new symmetry
of the distribution of the \emph{perturbed GUE
ensemble}. By this we mean the random matrix ensemble
of the form $H=G+\mathrm{diag}(a_1,\ldots,a_N)$, where $G$ is an $N\times N$
GUE random matrix, to which we add a fixed diagonal matrix. 
This model is often also called \emph{GUE with external source}.
We refer to 
\cite{desrosiers2006asymptotic},
\cite{adler2013random},
\cite{Ferrari2014PerturbedGUE}
and references therein for the history 
of the perturbed ensemble and various asymptotic results.
(In fact, below we consider a slightly more general version of the matrix
model involving a time-dependent rescaling; this version is suitable for 
the application to reflected Brownian motions.)

\medskip

The unperturbed GUE random matrix, corresponding to $a_i\equiv 0$,
is unitary invariant
in the sense that there is 
equality in distribution
$G\stackrel{d}{=}UGU^*$ for 
any fixed $N\times N$ unitary matrix $U$.
This implies that the distribution of the eigenvalues of 
$H$ is \emph{symmetric} in the perturbation parameters $a_1,\ldots,a_N $.
The overall goal of the paper is to \textbf{explore probabilistic
consequences of this symmetry property}.

\medskip

Together with the eigenvalues $\lambda^N=(\lambda^N_N\le \ldots\le\lambda^N_1 )$,
$\lambda^N_i\in \mathbb{R}$,
of the full matrix $H=[h_{ij}]_{i,j=1}^{N}$, 
one can also
consider its
\emph{corners process},\footnote{Also called
\emph{minors process} in the literature, 
cf. \cite{johansson2006eigenvalues}.}
that is, the interlacing collection of eigenvalues
of the principal corners
$[h_{ij}]_{i,j=1}^{k}$ of $H$ for all $k=1,2,\ldots,N $.
(See \Cref{fig:interlace} for an illustration.)
The distribution of the corners process of $H$ is 
\emph{not symmetric} in the parameters $a_i$.
Moreover, assuming that the $a_i$'s are 
all distinct, there are $N!$ \emph{different probability
distributions} on 
interlacing collections of eigenvalues
at $N$ levels. 

\begin{figure}[htpb]
	\centering
	\includegraphics[width=.55\textwidth]{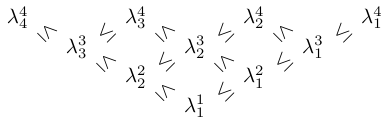}
	\caption{Interlacing array of eigenvalues
	of all principal corners of a $4\times 4$ matrix.}
	\label{fig:interlace}
\end{figure}

\medskip

In this paper we present explicit couplings 
between these $N!$ distributions, 
by showing that each nearest neighbour transposition
$a_k\leftrightarrow a_{k+1}$, $k=1,\ldots,N-1 $, of the parameters
is equivalent in distribution to a 
rather simple \emph{Markov swap operator}
$\mathscr{S}^{a_k-a_{k+1}}_k$. 
This swap operator
randomly changes the entries $\lambda^k_i$
on the $k$-th level of the array given the 
two adjacent levels $\lambda^{k-1},\lambda^{k+1}$,
while leaving all other entries intact.
If $a_k>a_{k+1}$, 
$\mathscr{S}^{a_k-a_{k+1}}_k$
is realized as an independent collection of instantaneous
exponential type jumps of each $\lambda^k_i$
to the left:\footnote{Here and below we use the standard notation
$A\vee B=\max(A,B)$, $A\wedge B=\min (A,B)$ for $A,B\in \mathbb{R}$.}
\begin{equation*}
    	\lambda^{k}_i\mapsto \nu^k_i:=
	\lambda^{k+1}_{i+1}\vee \lambda^{k-1}_{i}
	+
	\mathscr{E}_{a_k-a_{k+1}}^i\wedge
	\bigl(
		\lambda^k_i-
		\lambda^{k+1}_{i+1}\vee \lambda^{k-1}_{i}
	\bigr),
	\qquad i=1,\ldots,k ,
\end{equation*}
where $\mathscr{E}_{a_k-a_{k+1}}^{i}$'s are independent exponential random variables
with parameter $a_k-a_{k+1}$ (and mean $1 / (a_k-a_{k+1})$).
Here by agreement, $\lambda^{k-1}_k=-\infty$.
In particular, these left jumps are constrained by the interlacing.
For $a_k<a_{k+1}$, the same jumps are performed to the right in a symmetric way:
\begin{equation*}
	\lambda^{k}_i\mapsto \mu^k_i:=
	\lambda^{k+1}_{i}\wedge \lambda^{k-1}_{i-1}
	-
	\mathscr{E}_{a_{k+1}-a_{k}}^i\wedge
	\bigl(
		\lambda^{k+1}_{i}\wedge\lambda^{k-1}_{i-1}
		-
		\lambda^k_i
	\bigr),
	\qquad i=1,\ldots,k ,
\end{equation*}
where, by agreement, $\lambda^{k-1}_0=+\infty$.
Finally, if $a_k=a_{k+1}$, then $\mathscr{S}_k^{a_k-a_{k+1}}$ is the identity operation.

\begin{theorem}[Follows from \Cref{thm:action_of_level_swap_on_Gibbs} below]
	\label{thm:main_swap}
	Assume that $a_k\ne a_{k+1}$. Then the action of the Markov operator
	$\mathscr{S}^{a_k-a_{k+1}}_k$ (with left jumps for $a_k>a_{k+1}$, and right jumps otherwise)
	turns the corners distribution of $G+\mathrm{diag}(a_1,\ldots,a_k,a_{k+1},\ldots,a_N )$
	into the one of $G+\mathrm{diag}(a_1,\ldots,a_{k+1},a_{k},\ldots,a_N )$,
	where $G$ is the $N\times N$ GUE random matrix.
\end{theorem}

We establish this theorem by relying on 
a perturbed Gibbs structure of the corners distribution of the matrix $H$.
Namely, it is well-known that in the unperturbed case, the conditional distribution
of the eigenvalues $\lambda^k_i$, $1\le i\le k\le N-1$,
given $\lambda^N$, 
is uniform on the polytope defined by all the interlacing inequalities
(known as the Gelfand-Tsetlin polytope).
In the perturbed case, the Gibbs structure should be deformed in a certain
way by means of the 
parameters $a_i$ (see \Cref{sub:Gibbs_measures}). 
The coupling follows by considering
the conditional distribution
of $\lambda^k$ given two adjacent levels $\lambda^{k\pm 1}$,
which reduces to a collection of 
independent exponential random variables confined to the 
corresponding intervals. Producing a suitable Markov swap
operator for a single such variable (see \Cref{prop:elementary_swap} below),
we arrive at the result of  
\Cref{thm:main_swap}.

\begin{remark}
	Applied twice to $G+\mathrm{diag}(a_1,\ldots,a_N )$, the Markov swap operator 
	from \Cref{thm:main_swap} 
	returns to the same distribution.
	That it, the composition of 
	$\mathscr{S}_k^{a_{k}-a_{k+1}}$
	and 
	$\mathscr{S}_k^{a_{k+1}-a_{k}}$
	does not change the distribution of 
	$G+\mathrm{diag}(a_1,\ldots,a_N )$.
	However, this composition is \emph{not} an identity transformation: two random jumps return a particle to the
	original location with probability~0.
\end{remark}

\subsection{Perturbation by an arithmetic progression}

The perturbed GUE corners distributions 
are compatible for various $N$, and so one can define
the corresponding perturbed GUE corners
distribution on infinite interlacing arrays. 
It depends on an infinite parameter sequence $\mathbf{a}=\{a_i \}_{i\in \mathbb{Z}_{\ge1}}$.
One particular interesting case is when the perturbation
parameters form an arithmetic progression $a_i=-(i-1)\alpha$, where $\alpha>0$.
Swapping $a_1$ with $a_2$, then $a_1$ with $a_3$, and so on all the way to infinity
leads to an additive shift in the perturbation parameters, 
which is equivalent in distribution to a global 
shift:

\begin{theorem}[\Cref{thm:shift_for_alpha_GUE} below]
	\label{thm:main_theorem_shift}
	The action of a sequence of left exponential jumps 
	(where the parameter at level $k$ is taken to be $k\alpha$), from level $1$
	up to infinity, is equivalent in distribution\footnote{Here and below 
	by saying that two operations are ``equivalent in distribution'' we mean that the random elements 
	resulting from both these
	operations, applied to the same initial random element, have the same distribution.}
	to shifting all the 
	elements of the interlacing array by $\alpha$ to the left.
\end{theorem}

\subsection{Shifting of reflected Brownian motions}

Let $\alpha>0$, and let
let $X_k(t)$, $k=1,2,\ldots $, be reflected Brownian motions constructed as follows.
First, $X_1(t)$ is the standard driftless Brownian motion started from $0$.
Inductively, let $X_k(t)$, $k=2,3,\ldots $,
be a new independent Brownian motion with drift $-(k-1)\alpha$,
and reflected down off of $X_{k-1}(t)$ by means of subtracting 
local time when $X_{k}=X_{k-1}$. For example,
\begin{equation*}
	X_2(t)=X_2^\circ(t)-L_{1,2}(t),
\end{equation*}
where $X_2^\circ(t)$ is the standard Brownian motion, and 
$L_{1,2}(t)=\int_0^t \mathbf{1}_{X_1(s)=X_2(s)}\,dL_{1,2}(s)$
is the continuous non-decreasing process which increases only 
at times when $X_1(s)=X_2(s)$
(in other words, it
is twice the semimartingale local time of $X_1-X_2$ at zero.
We refer to \cite{warren2005dyson},
\cite{Ferrari2014PerturbedGUE}
for further details on the reflection mechanism, and for an explanation on how 
to start all these reflected processes from 
zero (which formally results in infinitely many collisions in finite time).
Almost surely we have $X_1(t)\ge X_2(t)\ge X_3(t) \ldots $ for all $t$.

Fix $t$ and define 
\begin{equation}
	\label{eq:BM_jumps}
	X'_k(t):=X_{k+1}(t)+\mathscr{E}_{k\alpha}\wedge
	(X_{k}(t)-X_{k+1}(t)),\qquad  k=1,2,\ldots ,
\end{equation}
where $\mathscr{E}_{k\alpha}$, $k=1,2,\ldots $, are independent exponential 
random variables with parameters $k\alpha$ (and mean $1 / (k\alpha)$).

\begin{theorem}
	\label{thm:main_theorem_Brownian}
	For each fixed $t$, we have equality of joint distributions
	\begin{equation*}
		\{X_k'(t)\}_{k\in \mathbb{Z}_{\ge1}}
		\stackrel{d}{=}
		\{X_k(t)-\alpha t \}
		_{k\in \mathbb{Z}_{\ge1}}.
	\end{equation*}
\end{theorem}

In particular, $X_1'(t)=X_2(t)+\mathscr{E}_{\alpha}\wedge(X_1(t)-X_2(t))$
is a normal random variable with mean $(-\alpha t)$ and variance $t$.
To the best of our knowledge, even this result for two processes
(one a usual Brownian motion, and one reflected off it)
is new.

\Cref{thm:main_theorem_Brownian} follows from \Cref{thm:main_theorem_shift}
combined with the connection between the 
reflected 
drifted Brownian motions 
and the perturbed GUE corners process
due to \cite{Ferrari2014PerturbedGUE}.
We recall this connection in detail in \Cref{sub:Brownian} below,
and prove \Cref{thm:main_theorem_Brownian}
in the end of \Cref{sec:proofs_5}.

\medskip

As stated, \Cref{thm:main_theorem_Brownian} assumes that the time $t$ is fixed. 
Indeed, naively taking
independent exponential shifts at different times $t$ would
make the functions $t\mapsto X_k'(t)$ discontinuous. 
It is interesting to see whether a 
stochastic process version
of \Cref{thm:main_theorem_Brownian} holds:
\begin{question}
	Is it possible to construct a Markov operator 
	on whole trajectories $t\mapsto \{X_k(t) \}_{k\in \mathbb{Z}_{\ge1}}$
	which is equivalent in distribution to a shift of 
	reflected Brownian motions as stochastic processes?
\end{question}
Presumably, if such a Markov operator on processes exists,
then its construction could be accomplished using the 
sticky Brownian motion,\footnote{We are grateful to Jon Warren
(personal communication) for
suggesting this connection.}
as exponential random variables arise
in the study of this process, e.g., see Theorem 1 in
\cite{warren1997branching}. 
It seems plausible that
the difference process $t\mapsto X_1(t)-X_1'(t)\ge0$ itself
could be distributed as the sticky Brownian motion, as the single-time distributions
coincide thanks to the results of \cite{warren1997branching} and
\cite[Proposition 14]{howitt_warren2009dynamics}. 
However, it is less clear how to extend this idea to
all differences $t\mapsto X_k(t)-X_k'(t)\ge0$.

\subsection{Related discrete model}
\label{sub:discrete_remarks}

The results of this paper might be 
viewed as a random matrix limit of the 
ones from the recent work
\cite{PetrovSaenz2019backTASEP}.
There, similar Markov swap operators were considered 
on discrete interlacing arrays as in \Cref{fig:interlace}.
A combination of these swap operators together with a 
certain Poisson-type limit
(cf. \Cref{sub:alpha0limit} below)
has lead to a Markov chain
on distributions of TASEP (totally asymmetric simple exclusion process)
which decreases the time parameter. 
The shifting result for reflected drifted Brownian motions 
(\Cref{thm:main_theorem_Brownian})
may be viewed as a certain analogue of the TASEP reversal property.
In the Brownian case, instead of decreasing the time, the 
exponential jumps lead to a deterministic shift.

It should be pointed out that 
even though the discrete stochastic systems in 
\cite{PetrovSaenz2019backTASEP}
converge to the reflected Brownian motions \cite{GorinShkolnikov2012}
(and \cite{Ferrari2014PerturbedGUE} in the drifted case),
here we do not rely on this convergence or the results of
\cite{PetrovSaenz2019backTASEP}. Instead we obtain the results 
independently using basic mechanisms related to the (perturbed) Gibbs property.

\subsection{Unperturbed case}
\label{sub:alpha0limit}

In the arithmetic progression setting
$a_i=-(i-1)\alpha$ with $\alpha>0$,
when $\alpha\searrow0$, the perturbed GUE corners process 
of $H=G+\mathrm{diag}(0,-\alpha,-2\alpha,\ldots )$
becomes the usual GUE corners process,
and the system 
of reflected Brownian motions $\{X_k(t)\}_{k\in \mathbb{Z}_{\ge1}}$
becomes driftless.
It would be very interesting to see whether the Markov operators
considered in the present paper 
have meaningful limits as $\alpha\searrow0$.
However, this limit presents certain immediate issues 
which we discuss now. 

\medskip

For simplicity, consider the 
Brownian motion setup. Fix $t>0$ and suppress this parameter in the notation.
As $\alpha\to0$, 
the Markov operator $X_k\mapsto X_k'$ \eqref{eq:BM_jumps}
turns into the (deterministic)
identity operator $X_k\mapsto X_k$.
Indeed, this is because 
$\mathrm{Prob}(\mathscr{E}_{k\alpha}>x)=e^{-k\alpha x}\sim 1-\alpha k x$ for all $k$ and $x$,
and so the minimum in \eqref{eq:BM_jumps}
is equal to $X_k-X_{k+1}$ with probability of order 
$1-O(\alpha)$.
Arguing similarly to the 
discrete case considered in \cite[Section 6]{PetrovSaenz2019backTASEP},
one can apply the map \eqref{eq:BM_jumps} 
a large number $\lfloor \tau/\alpha \rfloor $
of times, where $\tau\in \mathbb{R}_{>0}$ is the scaled time.

Taking a Poisson-type limit should lead to a continuous time 
Markov process (with $\tau$ as the new time parameter)
under which $X_k$ has an exponential clock of rate $k(X_k-X_{k+1})$,
and when the clock rings, $X_k$ instantaneously jumps
into $X_k'$ selected uniformly from $[X_{k+1},X_k]$.
This jumping mechanism is known as the \emph{Hammersley process}
\cite{hammersley1972few}, \cite{aldous1995hammersley}.
However, applying this continuous time 
jumping process to the whole
system $\{X_k \}_{k\in \mathbb{Z}_{\ge1}}$
is problematic, as it leads to \emph{infinitely many jumps in finite time}
due to the growing jump rates $k(X_k-X_{k+1})$ as $k\to\infty$. 
Moreover, under this hypothetical process $X_k$ would 
depend on all $X_j$ for $j>k$, and so one cannot simply restrict
the dynamics to finitely many particles where it would make sense.

On the other hand, by \Cref{thm:main_theorem_Brownian},
the hypothetical continuous time dynamics should be equivalent in distribution
to a deterministic shift of the (driftless) reflected Brownian motions
by 
$-\alpha t\lfloor \tau/\alpha \rfloor 
\sim -t\tau$. 
It is reasonable to expect that such a deterministic shift of infinitely
many $X_k$'s
cannot be achieved
only by finitely many jumps in finite time.
To summarize,

\begin{question}
	Do there exist well-defined $\alpha\searrow0$ limits of the Markov operators acting 
	on the GUE corners process perturbed by an arithmetic progression $a_i=-(i-1)\alpha$
	or on the reflected drifted Brownian motions? 
	These hypothetical limits should act on (much more studied)
	unperturbed GUE corners process and driftless reflected Brownian motions.
\end{question}

\subsection{Acknowledgments}

We are grateful to
Krzysztof Burdzy, Christian Gromoll,  
Grigori Olshanski
and Jon Warren,
for 
helpful discussions.
We acknowledge the hospitality of
the organizers of the 
Workshop on Classical and Quantum Integrable Systems (CQIS-2019)
at 
Euler Institute, Saint Petersburg,
where this work was started.
Both authors were partially supported by
the NSF grant DMS-1664617.

\section{Perturbed GUE corners process}
\label{sec:perturbed_GUE_corners_recall}

This section is preliminary. We 
recall the 
perturbed GUE corners process \cite{Ferrari2014PerturbedGUE}
(also called the GUE corners process with external source \cite{adler2013random}),
and its connection
to reflected Brownian motions with drifts.
The original, unperturbed GUE corners process is due to 
\cite{johansson2006eigenvalues},
\cite{johansson2006eigenvaluesErratum},
and it was linked to driftless reflected Brownian motions in 
\cite{warren2005dyson}.

\subsection{Matrix model}
\label{sub:matrix_model}

Take a 
time parameter $t>0$ and 
an infinite sequence of parameters
\begin{equation*}
	\mathbf{a}=(a_1,a_2,\ldots ),\qquad a_i\in \mathbb{R}.
\end{equation*}
Unless otherwise indicated, we assume that the parameters
$a_i$ are pairwise distinct.
Consider a random matrix $H=t^{1/2}\cdot G+t\cdot\mathrm{diag}\left( \mathbf{a} \right)$ of infinite size
with entries:
\begin{equation*}
	H_{kl}=
	\begin{cases}
		t^{1/2}g_{kk}+t\mu_k,& k=l;\\
		t^{1/2}g_{kl},&k<l;\\
		t^{1/2}\overline{g}_{lk},&k>l.
	\end{cases}
\end{equation*}
Here $g_{kk}$ are independent real standard normal 
random variables, and $g_{kl}$ are 
independent complex standard normal random variables
(that is, their real and imaginary parts are independent real 
normal random variables each with mean $0$ and variance $\frac{1}{2}$).
The matrix $H$ is Hermitian.

For each $m\in \mathbb{Z}_{\ge1}$, take the $m\times m$ principal corner 
$[H_{kl}]_{1\le k,l\le m}$
of the infinite matrix
$H$. Let $\lambda^m=(\lambda^{m}_1\ge \ldots\ge\lambda^{m}_{m})$,
$\lambda^{m}_{i}\in \mathbb{R}$,
be its eigenvalues. At adjacent levels, the eigenvalues \emph{interlace}
(notation $\lambda^m\prec \lambda^{m+1}$):
\begin{equation}
	\label{eq:interlace}
	\lambda^{m+1}_{m+1}\le \lambda^{m}_{m}\le \lambda^{m+1}_{m}\le 
	\lambda^{m}_{m-1}\le
	\ldots\le \lambda^{m+1}_2\le \lambda^{m}_1\le\lambda^{m+1}_1 .
\end{equation}
We call the joint distribution of 
all $\{\lambda^{k}_j\}_{1\le j\le k<\infty}$
the \emph{perturbed GUE corners process}.

\subsection{Joint eigenvalue density}

A standard application of the Harish-Chandra-Itsykson-Zuber
integral shows that the joint eigenvalue density of 
$\{\lambda^{N}_i\}_{i=1}^{N}$
at a fixed level $N$ is 
given by 
\begin{equation}
	\label{eq:lambda_N_density}
	\mathsf{Density}(\lambda^N)=
	\mathrm{const}\times
	\det\left[ \exp\left\{ 
			-\frac{(\lambda_i^{N}-t a_j)^2}{2t}
	\right\} \right]_{i,j=1}^{N}
	\frac{\mathsf{V}(\lambda_1^{N},\ldots,\lambda^{N}_N )}{\mathsf{V}(a_1,\ldots,a_N )},
\end{equation}
where the normalizing constant does not depend on $a_1,\ldots,a_N$.
Here and throughout the paper we use the notation 
\begin{equation*}
	\mathsf{V}(b_1,\ldots,b_N )=
	\prod_{1\le i<j\le N}(b_i-b_j)
\end{equation*}
for the Vandermonde determinant.

Observe from \eqref{eq:lambda_N_density}
that the distribution of
$\{ \lambda^{N}_j \}_{i=1}^{N}$
depends on the parameters $a_i$ in a symmetric way.
This should indeed be the case, since the distribution of the 
eigenvalues of 
the $N\times N$ matrix
$t^{1/2}\, G_{N\times N}+t\,\mathrm{diag}(a_1,\ldots,a_N )$
does not depend on the order of the $a_i$'s due to the unitary
invariance of $G_{N\times N}$.
The main goal of this paper is to 
explore this distributional symmetry from a Markov
operator point of view. For this, we
will need the joint distribution of eigenvalues of all
corners:

\begin{proposition}[{\cite[Proposition 2.3]{Ferrari2014PerturbedGUE}}]
	\label{prop:FF_joint_density}
	The joint density of the eigenvalues $\{\lambda^{k}_j\}_{1\le j\le k\le N}$
	at the first $N$ levels, where $N\in \mathbb{Z}_{\ge1}$
	is arbitrary, has the following form:
	\begin{equation}
		\label{eq:pert_GUE_joint}
			\mathrm{const}\times
			\mathsf{V}(\lambda_1^N,\ldots,\lambda^N_N )
			\prod_{i=1}^{N}
			e^{-ta_i^2/2-
			(\lambda^{N}_i)^2/(2t)}
			\exp\Bigl\{ |\lambda^N|\,a_N+\sum_{k=1}^{N-1}|\lambda^k|\,(a_k-a_{k+1}) \Bigr\}
	\end{equation}
	where 
	we use the notation
	$|\lambda^k|:=\lambda^k_1+\lambda^k_2+\ldots+\lambda^k_k $,
	and the normalizing constant does not depend on $a_1,\ldots,a_N$.
\end{proposition}

\subsection{Reflected Brownian motions}
\label{sub:Brownian}

Fix a perturbation sequence $\mathbf{a}=\{a_i \}_{i\in \mathbb{Z}_{\ge1}}$.
Consider a family of interacting Brownian motions 
$B^k_j$, $1\le j\le k<\infty$, such that:
\begin{itemize}
	\item All processes start from zero.
	\item The processes $B^k_j$, $j=1,\ldots,k $
		have the same drift $a_k$.
	\item The evolution of each $B^k_j$
		does not depend on any of the $B^l_i$'s with $l>i$.
	\item The processes interact only through their local times.
		That is, when the processes are sufficiently far apart,
		each $B^k_j$ behaves as an independent Brownian motion 
		with drift $a_k$.
	\item Each $B^k_j$ belongs to the segment
		$[B^{k-1}_j,B^{k-1}_{j-1}]$\footnote{If
			one or both ends of the segment are not defined,
			they should be replaced with infinity of appropriate sign.}
		and reflects off 
		both 
		$B^{k-1}_j$ and $B^{k-1}_{j-1}$.
		Therefore, at each time $t$, the 
		random variables
		$\{B^k_j(t)\}_{1\le j\le k<\infty}$ almost surely form
		an interlacing array as in \Cref{fig:interlace}.
\end{itemize}

We refer to 
\cite[Section 4]{Ferrari2014PerturbedGUE}
(and \cite{warren2005dyson} in the driftless case)
for details on the reflection mechanism.

\begin{proposition}[{\cite[Theorem 2]{Ferrari2014PerturbedGUE}}]
	\label{prop:connection_to_reflected_BM}
	At each fixed time moment $t\in \mathbb{R}_{\ge0}$, we have equality
	of joint distributions of two infinite interlacing arrays:
	\begin{equation*}
		\{B^k_j(t)\}_{1\le j\le k<\infty}
		\stackrel{d}{=}
		\{\lambda^k_j\}_{1\le j\le k<\infty},
	\end{equation*}
	where the right-hand side is the perturbed GUE
	corners process with the same time parameter $t$
	and perturbation sequence $\mathbf{a}$.
\end{proposition}

\section{Gibbs measures}

In this section we place the perturbed GUE corners
process into a wider family of 
Gibbs measures on interlacing arrays.

\subsection{Gibbs property and harmonic functions}
\label{sub:Gibbs_measures}

A measure on infinite interlacing arrays
$\{\lambda^k_j \}_{1\le j\le k<\infty}$
(satisfying inequalities \eqref{eq:interlace} between any two consecutive levels)
is called \emph{$\mathbf{a}$-Gibbs} if for each 
$N$ and any fixed configuration $\lambda^N$ at level $N$, the 
density of the conditional 
distribution of all the lower entries of the array has the form
\begin{equation}
	\label{eq:a_Gibbs}
	\begin{split}
		&
		\mathsf{Density}(\lambda^1,\ldots,\lambda^{N-1}\mid \lambda^N )=
		\frac{\mathsf{V}(a_1,\ldots,a_N )}{\det[\exp\{ a_i \lambda^N_j \}]_{i,j=1}^N}
		\\&\hspace{70pt}\times
		\exp\left\{ |\lambda^N|\,a_N+\sum_{k=1}^{N-1}|\lambda^k|\,(a_k-a_{k+1}) \right\}
		\mathbf{1}_{\lambda^1\prec \lambda^2\prec \ldots\prec\lambda^{N-1}\prec\lambda^N }
	\end{split}
\end{equation}
(if some of the $\lambda_i^N$'s are equal, the density 
would have delta components and formula \eqref{eq:a_Gibbs}
should be understood in a limiting sense).
Here and below $\mathbf{1}_{B}$ stands for the indicator of an event $B$.
\Cref{prop:FF_joint_density} implies that
the perturbed GUE 
corners process is an example of an 
$\mathbf{a}$-Gibbs measure.
\begin{remark}
	The fact that the density \eqref{eq:a_Gibbs} integrates to $1$
	in $\lambda^1,\ldots,\lambda^{N-1} $ can be checked by induction on $N$.
\end{remark}
\begin{remark}
	When $a_i\equiv a$ are all equal to each other, 
	the 
	$\mathbf{a}$-Gibbs property becomes the usual Gibbs
	property, with \eqref{eq:a_Gibbs} replaced by the uniform conditioning
	provided that the configurations $\lambda^1,\ldots,\lambda^{N-1},\lambda^N $ interlace.
	A classification of uniform Gibbs measures on interlacing arrays is
	due to \cite{OlVer1996}.
	In fact, performing a suitable exponential change of variables,
	one can see that when $\mathbf{a}$ is an arithmetic progression,
	the space of
	$\mathbf{a}$-Gibbs measures is essentially the same as in the uniform case.
	This is somewhat parallel to how the two-sided $q$-Gelfand-Tsetlin graph
	degenerates to the ``graph of spectra''
	\cite{gorin2016quantization}, \cite{olshanski2016extended}.
\end{remark}

To each $\mathbf{a}$-Gibbs
measure we can associate a family of $\mathbf{a}$-\emph{harmonic functions}
as follows:
\begin{equation}
	\label{eq:a_harmonic_defn}
	\varphi_N(\lambda^N):=
	\frac{\mathsf{V}(a_1,\ldots,a_N )}{\det[\exp\{a_i\lambda^N_j \}]_{i,j=1}^N}
	\,\mathsf{Density}(\lambda^N),\qquad N=1,2,\ldots, 
\end{equation}
where $\mathsf{Density}(\lambda^N)$ is the marginal density of $\lambda^N$.
The term ``harmonic function'' comes from the Vershik--Kerov 
theory of boundaries of branching graphs, cf. \cite{Kerov1998}.
Harmonicity means that 
the functions satisfy a version of a mean value theorem associated to a directed graph Laplacian.
In the context of random matrices the discrete graph is replaced by a suitable continuous
analogue, and the graph Laplacian becomes an integral operator.
In other words, 
since the $\mathbf{a}$-harmonic functions $\varphi_N$ for different $N$'s
come from the same Gibbs measure, they must be
consistent in the following sense:
\begin{lemma}
	\label{lemma:Gibbs_consistency}
	For all $N\ge2$ we have
	\begin{equation}
		\label{eq:Gibbs_consistency}
		\varphi_{N-1}(\lambda^{N-1})=
		\int_{\lambda^N\colon \lambda^N\succ \lambda^{N-1}}
		\varphi_N(\lambda^N)
		\,e^{a_N(|\lambda^N|-|\lambda^{N-1}|)}d\lambda^N.
	\end{equation}
\end{lemma}
Identity \eqref{eq:Gibbs_consistency} should be 
viewed as a version of the mean value theorem, as discussed before
\Cref{lemma:Gibbs_consistency}.
\begin{proof}[Proof of \Cref{lemma:Gibbs_consistency}]
	The claim follows by writing down the joint distribution of
	$\lambda^1,\ldots,\lambda^N$ through $\varphi_N$ and the 
	conditional distribution \eqref{eq:a_Gibbs},
	and then integrating out $\lambda^1,\ldots,\lambda^{N-2}$ 
	(this produces the factor 
	$\mathsf{V}(a_1,\ldots,a_{N-1} )/\det[\exp\{ a_i\lambda^{N-1}_j \}]_{i,j=1}^{N-1}$)
	and $\lambda^{N}$
	to get the marginal density of $\lambda^{N-1}$.
	The resulting marginal density is expressed through
	$\varphi_{N-1}$ via \eqref{eq:a_harmonic_defn}, 
	which yields the result.
\end{proof}

\begin{lemma}
	\label{prop:Gibbs_symmetry}
	For an $\mathbf{a}$-Gibbs measure,
	let each $\varphi_k$ depend on $a_1,\ldots,a_k $ in a symmetric way.
	Then the
	distribution of $\lambda^k$, where $k\in \mathbb{Z}_{\ge1}$ is fixed,
	depends on the parameters $a_1,\ldots,a_k$ in a symmetric way, too.
\end{lemma}
\begin{proof}
	An immediate consequence of \eqref{eq:a_harmonic_defn}.
\end{proof}

\begin{proposition}
	\label{prop:uniquely}
	Any $\mathbf{a}$-Gibbs measure
	is uniquely determined by the 
	corresponding family of 
	$\mathbf{a}$-harmonic functions 
	$\left\{ \varphi_N \right\}_{N\in \mathbb{Z}_{\ge1}}$.
\end{proposition}
\begin{proof}
	Follows from the Kolmogorov extension theorem.
\end{proof}

Let us emphasize that the 
results of this subsection (\Cref{lemma:Gibbs_consistency,prop:Gibbs_symmetry} and \Cref{prop:uniquely})
are valid not only for the perturbed GUE corners process 
(which, as we see next, is an example of an $\mathbf{a}$-Gibbs measure),
but hold in the full generality for any $\mathbf{a}$-Gibbs measure.

\subsection{Perturbed GUE corners as a Gibbs measure}

One readily sees that 
for the perturbed GUE corners process
we have the following harmonic functions:
\begin{equation}
	\label{eq:pertGUE}
	\varphi_N^{\mathrm{pertGUE}(\mathbf{a};t)}(\lambda^N)
	=\mathrm{const}\times
	\mathsf{V}(\lambda_1^N,\ldots,\lambda^N_N )
	\prod_{i=1}^{N}
	e^{-ta_i^2/2-
	(\lambda^{N}_i)^2/(2t)},\qquad 
	N=1,2,\ldots ,
\end{equation}
where the constant is the same as in 
\eqref{eq:pert_GUE_joint}
and does not depend on the $a_j$'s.
One readily checks that the $\mathbf{a}$-Gibbs property 
(\Cref{lemma:Gibbs_consistency})
for the perturbed GUE corners process
is equivalent to the well-known integral identity for the Vandermonde determinants:
\begin{equation}
	\label{eq:pertGUE_Gibbs_property}
	\begin{split}
	&\mathsf{V}(\lambda_1^{N-1},\ldots,\lambda^{N-1}_{N-1} )
	\prod_{i=1}^{N-1}
	e^{-
	(\lambda^{N-1}_i-a_N t)^2/(2t)}
	\\&\hspace{70pt}=
	\mathrm{const}\times
	\int_{\lambda^N\colon \lambda^N\succ \lambda^{N-1}}
	\mathsf{V}(\lambda_1^N,\ldots,\lambda^N_N )
	\prod_{i=1}^{N}
	e^{-
	(\lambda^{N}_i-a_N t)^2/(2t)}
	d\lambda^N.
	\end{split}
\end{equation}
where the constant does not depend on the $a_j$'s.
The shift by $a_Nt$ in the exponents in both sides
by changing the variables in the integral and renaming the $\lambda^{N-1}_i$'s,
can also be removed (or replaced with any other shift $bt$)
since the Vandermonde is translation invariant.

In particular, \eqref{eq:pert_GUE_joint}
together with \Cref{prop:Gibbs_symmetry}
implies the symmetry (as in this lemma)
of the perturbed GUE corners distribution with respect to the parameters $a_i$.

\section{Swap operators via exponential jumps}
\label{sec:swap_operators_thm44}

In this section we explore the Gibbs
property and prove \Cref{thm:main_swap}
on Markov swap operators.

\subsection{Confined exponential distribution}

Let $c<d$ and $\alpha$ be real numbers.
Let us call a random variable
on $(c,d)$ with probability density 
\begin{equation*}
	\frac{\alpha}{e^{d\alpha}-e^{c\alpha}}\,e^{\alpha x},\qquad x\in(c,d),
\end{equation*}
an \emph{exponential random variable confined
to the segment $(c,d)$}, notation $E_\alpha(c,d)$.
Note that this definition makes sense
regardless of the sign of $\alpha$ (in contrast with the case when the 
interval $(c,d)$ is half-infinite).
If $\alpha=0$, then $E_0(c,d)$ is
simply the uniform random variable on $(c,d)$.

\subsection{Elementary Markov swap operator}

The next observation plays a key role:
\begin{proposition}
	\label{prop:elementary_swap}
	Take real numbers $c<d$ and $\alpha>0$. 
	Let $X$ be distributed as $E_\alpha(c,d)$,
	and $\mathscr{E}_\alpha\in(0,+\infty)$ be an independent usual
	exponential random variable with parameter $\alpha$
	(i.e., with density $\alpha e^{-\alpha y}$, $y>0$). Then
	the random variable
	\begin{equation}
		\label{eq:Y_as_function_of_X_elementary_swap}
		Y:=c+\mathcal{E}_\alpha\wedge(X-c)
	\end{equation}
	is distributed as $E_{-\alpha}(c,d)$.
\end{proposition}
\begin{proof}
	We have for the conditional distribution of $Y$ given
	$X=x$:
	\begin{equation}
		\label{eq:density_swap}
		\mathrm{Prob}\left( Y\in [y,y+\mathsf{d}y]\mid X=x \right)
		=
		\mathbf{1}_{x=y}e^{-\alpha(y-c)}+\alpha e^{-\alpha(y-c)}\mathsf{d}y,
		\qquad c\le y\le x.
	\end{equation}
	The distribution of $Y$ has an atom at $y=x$ 
	(coming from the event $\mathscr{E}_\alpha>X-c$ in
	\eqref{eq:Y_as_function_of_X_elementary_swap})
	and an absolutely continuous
	part on $(0,x)$.
	The overall density of $Y$ in the variable $y$
	is obtained from the following
	integral:
	\begin{equation*}
		\begin{split}
			&\int_{y}^{d}
			\frac{\alpha}{e^{d\alpha}-e^{c\alpha}}\,e^{\alpha x}
			\,\mathrm{Prob}\left( Y\in [y,y+\mathsf{d}y]\mid X=x \right)\mathsf{d}x
			\\&\hspace{40pt}=
			\frac{\alpha}{e^{d\alpha}-e^{c\alpha}}\,e^{\alpha y}
			e^{-\alpha(y-c)}+
			\frac{\alpha}{e^{d\alpha}-e^{c\alpha}}\,
			\alpha e^{-\alpha(y-c)}
			\int_{y}^{d}
			e^{\alpha x}\mathsf{d}x
			\\&\hspace{40pt}=
			\frac{-\alpha}{e^{-d\alpha}-e^{-c\alpha}}\,e^{-\alpha y},
		\end{split}
	\end{equation*}
	which completes the proof.
\end{proof}

We will view the operation of passing from $X$ to $Y$
as in \eqref{eq:Y_as_function_of_X_elementary_swap}
as a one-step Markov transition operator. One can think that the 
``particle'' $X\in (c,d)$
jumps left into the new location $Y$,
by means of the new exponential random variable $\mathscr{E}_\alpha$.
Note that the new location $Y$ depends only on $X$ and not on the 
right endpoint $d$ of the interval. 
We call this Markov transition operator 
the \emph{elementary swap operator} and denote it by $S^\alpha$.
This operator acts on 
distributions (in our case, densities) as $\mathsf{Density}_Y=\mathsf{Density}_X \, S^\alpha$.

The swap operator $S^\alpha$ is analogous to the 
jump operator $L_\alpha$ in the discrete situation
considered in \cite[Section 4]{PetrovSaenz2019backTASEP}.
Let us make a number of remarks.

\begin{remark}
	\label{rmk:elementary_swap_operators_remarks}
	\begin{enumerate}		
		\item 
			When $\alpha=0$, the swap operator $S^\alpha$ 
			should be understood as the identity map, which is evident from
			\eqref{eq:density_swap}.
		\item 
			For $\alpha=-\beta<0$, algebraic manipulations in the proof of 
			\Cref{prop:elementary_swap} make sense,
			but the new random variable $Y$ obtained by applying $S^{-\beta}$ to 
			$X\sim E_{-\beta}(c,d)$
			does not admit a probabilistic interpretation as in 
			\eqref{eq:Y_as_function_of_X_elementary_swap}.
		\item 
			Instead of applying $S^{-\beta}$ to $E_{-\beta}(c,d)$, 
			let us consider the operator which 
			moves $X$ to the right symmetrically to how $S^\alpha$ moves the 
			``particle'' $X$
			to the left.
			That is, this new operator acts as
			$Y'=d-\mathscr{E}_\beta\wedge(d-X)$, where $\mathscr{E}_\beta$ is
			an independent exponential random variable.
			One can show similarly to \Cref{prop:elementary_swap}
			that if $X\sim E_{-\beta}(c,d)$, then $Y'\sim E_{\beta}(c,d)$.
			All our results for Markov operators built from the left jumps
			$S^\alpha$
			have straightforward analogues for these right jumping operators,
			and so we will only focus on the left jumps in the paper.
	\end{enumerate}
\end{remark}

\subsection{Swap operator for Gibbs measures}

Let us fix 
a perturbation sequence $\mathbf{a}$,
and let $\{\lambda^m_j \}_{1\le j\le m<\infty}$
be a random interlacing array distributed
according to some $\mathbf{a}$-Gibbs measure
(for example, the 
perturbed GUE corners process 
with an arbitrary time parameter $t\ge0$).

Next, fix a level $k\in \mathbb{Z}_{\ge1}$, and
consider the conditional distribution of $\lambda^k$ given 
the two adjacent levels $\lambda^{k-1}$, $\lambda^{k+1}$
(if $k=1$, the conditioning is only on $\lambda^2$).
From \eqref{eq:a_Gibbs} one readily sees that 
this conditional distribution takes the form
\begin{equation}
	\label{eq:conditional_local_distribution}
	\mathsf{Density}(\lambda^k\mid \lambda^{k-1},\lambda^{k+1})=
	\mathrm{const}\times
	\exp\left\{ \alpha(\lambda^k_1+\ldots+\lambda^k_k ) \right\}
	\mathbf{1}_{\lambda^{k-1}\prec\lambda^k\prec \lambda^{k+1}},
\end{equation}
where we have denoted $\alpha:=a_k-a_{k+1}$.
Equivalently, we can describe distribution 
\eqref{eq:conditional_local_distribution} as follows. 
\begin{proposition}
	\label{prop:independent_exponentials}
	The conditional distribution of $\lambda^k$ 
	given $\lambda^{k-1}$ and $\lambda^{k+1}$ is such that
	each $\lambda^k_i$, $i=1,\ldots,k $, is an independent
	random variable distributed as 
	\begin{equation}
		\label{eq:interval_for_lambda_k_i}
		E_\alpha
		\bigl(\lambda^{k+1}_{i+1}\vee \lambda^{k-1}_{i}
		,
		\lambda^{k+1}_i\wedge \lambda^{k-1}_{i-1}
		\bigr),
	\end{equation}
	where $\alpha=a_k-a_{k+1}$.
	(For $i=k$ we set $\lambda^{k-1}_k=-\infty$, 
	and for $i=1$ we set $\lambda^{k-1}_0=+\infty$,
	but
	both ends of the interval in \eqref{eq:interval_for_lambda_k_i} are
	always finite.)
\end{proposition}
\begin{proof}
	Readily follows from \eqref{eq:conditional_local_distribution}.
\end{proof}

Assume that $\alpha=a_k-a_{k+1}>0$,
and take an array $\{\lambda^{m}_j\}_{1\le j\le m<\infty}$
as above.
Let us define a new random interlacing array
$\{\nu^m_j\}_{1\le j\le m<+\infty}$ for which
$\nu^m_j=\lambda^m_j$ for all $m\ne k$, $j=1,\ldots,m $,
and such that 
\begin{equation}
	\label{eq:swap_operator_at_a_single_level}
	\nu^k_i:=
	\lambda^{k+1}_{i+1}\vee \lambda^{k-1}_{i}
	+
	\mathscr{E}_\alpha^i\wedge
	\bigl(
		\lambda^k_i-
		\lambda^{k+1}_{i+1}\vee \lambda^{k-1}_{i}
	\bigr)
	,\qquad 
	i=1,\ldots,k,
\end{equation}
where $\mathscr{E}^1_\alpha,\ldots,\mathscr{E}^k_\alpha $
are independent usual exponential random variables 
with parameter $\alpha$.
Note that almost surely we have
$\nu^k_i\le \lambda^k_i$, $i=1,\ldots,k $.

In other words, in \eqref{eq:swap_operator_at_a_single_level}
we independently apply the elementary swap operator $S^\alpha$
to each $\lambda^k_i$
which is confined to the corresponding
interval as in \Cref{prop:independent_exponentials}.
Denote this combination of the 
swap operators applied at level $k$ by 
$\mathscr{S}^\alpha_k$.
As in \Cref{rmk:elementary_swap_operators_remarks}, 
the Markov operator 
$\mathscr{S}^\alpha_k$ makes sense only for $\alpha>0$.

Let $\tau_k$ denote the elementary transposition $(k,k+1)$.
For a perturbation sequence $\mathbf{a}$, let 
$\tau_k\mathbf{a}=(a_1,\ldots,a_{k-1},a_{k+1},a_k,\ldots )$ be the 
permuted sequence.

\begin{theorem}[\Cref{thm:main_swap} in Introduction]
	\label{thm:action_of_level_swap_on_Gibbs}
	Take an $\mathbf{a}$-Gibbs measure
	for which each harmonic function $\varphi_N$
	depends on the parameters $a_1,\ldots,a_N $ in a symmetric way.
	If $a_k>a_{k+1}$, then the action of the Markov operator
	$\mathscr{S}^\alpha_{k}$ (with $\alpha=a_k-a_{k+1}$)
	on this $\mathbf{a}$-Gibbs measure
	results in a $\tau_k\mathbf{a}$-Gibbs measure
	which corresponds to harmonic functions modified as follows:
	\begin{equation}
		\label{eq:modification_under_harmonic_functions}
		\begin{split}
			\varphi_j'&=\varphi_j,\qquad j\ne k;\\
			\varphi_k'(\lambda^k)&=
			\int_{\lambda^{k+1}\colon \lambda^{k+1}\succ \lambda^k}
			\varphi_{k+1}(\lambda^{k+1})\,
			e^{a_k(|\lambda^{k+1}|-|\lambda^k|)}d\lambda^{k+1}.
		\end{split}
	\end{equation}
\end{theorem}
\begin{proof}
	Since the action of $\mathscr{S}^\alpha_k$ does not 
	change levels $j\ne k$ (and hence distributions
	of these levels), we clearly have $\varphi_j'=\varphi_j$ for $j\ne k$.

	Thus, it remains to show that under $\mathscr{S}^\alpha_k$ the 
	$\mathbf{a}$-Gibbs property becomes $\tau_k\mathbf{a}$-Gibbs.
	This can be seen by representing the 
	conditional distributions as 
	\begin{equation}
		\label{eq:modification_under_harmonic_functions_proof}
		\mathrm{Prob}(\lambda^1,\ldots,\lambda^k\mid \lambda^{k+1} )=
		\mathrm{Prob}(\lambda^1,\ldots,\lambda^{k-1}\mid \lambda^{k+1} )
		\cdot
		\mathrm{Prob}(\lambda^{k}\mid \lambda^{k-1},\lambda^{k+1}).
	\end{equation}
	Due to 
	\eqref{eq:a_Gibbs},
	the left-hand side depends on $a_1,\ldots,a_{k+1}$ in the following way:
	\begin{equation}
		\label{eq:new}
		f(a_1,\ldots,a_{k+1})
		\exp
		\left\{ a_{k+1}(|\lambda^{k+1}|-|\lambda^k|)+
		a_k(|\lambda^{k}|-|\lambda^{k-1}|) 
		+\ldots 
		\right\}
	\end{equation}
	where $f$ is symmetric in $a_1,\ldots,a_{k+1}$.
	One can readily check that 
	$\mathrm{Prob}(\lambda^1,\ldots,\lambda^{k-1}\mid \lambda^{k+1} )$
	depends on the parameters $a_k,a_{k+1}$ in a symmetric way, too.
	Indeed, this conditional distribution
	corresponds to integrating 
	\eqref{eq:a_Gibbs} (with $N=k+1$) over $\lambda^k$.
	The non-exponential prefactor in \eqref{eq:a_Gibbs} depending on $\lambda^{k+1}$
	is already symmetric, and for the exponential part we have
	\begin{equation}
		\label{eq:modification_under_harmonic_functions_proof1}
		\begin{split}
			&e^{a_{k+1}|\lambda^{k+1}|+\sum_{j=1}^{k-1}|\lambda^j|(a_j-a_{j+1})}
			\int 
			e^{|\lambda^k|(a_k-a_{k+1})}
			d\lambda^k
			\\&\hspace{40pt}=
			e^{a_{k+1}|\lambda^{k+1}|+\sum_{j=1}^{k-1}|\lambda^j|(a_j-a_{j+1})}
			\prod_{i=1}^{k}
			\frac{e^{\alpha(\lambda^{k+1}_{i+1}\vee \lambda^{k-1}_{i})}
			-
			e^{\alpha(\lambda^{k+1}_i\wedge\lambda^{k-1}_{i-1})}}{\alpha},
		\end{split}
	\end{equation}
	where we used the normalizing constant 
	for the confined exponential distribution, and $\alpha=a_k-a_{k+1}$.
	Swapping the parameters as $a_k\leftrightarrow a_{k+1}$
	brings 
	$\exp\left\{ -\alpha(|\lambda^{k+1}|+|\lambda^{k-1}|) \right\}$
	from the exponential factor in front of the product in
	\eqref{eq:modification_under_harmonic_functions_proof1}.
	This factor compensates
	the product of the expressions
	$\exp\left\{ \alpha\bigl(\lambda^{k+1}_{i+1}\vee \lambda^{k-1}_{i}+
	\lambda^{k+1}_i\wedge\lambda^{k-1}_{i-1}\bigr) \right\}$
	over $i=1,\ldots,k $, coming out of the product in 
	\eqref{eq:modification_under_harmonic_functions_proof1} 
	after the same swap.
	Thus, \eqref{eq:modification_under_harmonic_functions_proof1} is
	symmetric under $a_k\leftrightarrow a_{k+1}$.

	The action of $\mathscr{S}^\alpha_k$ affects only 
	the part 
	$\mathrm{Prob}(\lambda^{k}\mid \lambda^{k-1},\lambda^{k+1})$
	in the right-hand side of \eqref{eq:modification_under_harmonic_functions_proof}. 
	Before the action of $\mathscr{S}^\alpha_k$,
	each $\lambda^k_i$
	was distributed as $E_\alpha$ on the corresponding interval
	(see \Cref{prop:independent_exponentials}).
	By \Cref{prop:elementary_swap}, after the action of $\mathscr{S}^\alpha_k$,
	these random variables turn into the
	$E_{-\alpha}$'s, which corresponds to the $\tau_k\mathbf{a}$-Gibbs
	structure, see \eqref{eq:new}. Combining this with the symmetries in 
	\eqref{eq:modification_under_harmonic_functions_proof} described above
	and using \Cref{prop:Gibbs_symmetry} and \Cref{prop:uniquely},
	we arrive at the claim.
\end{proof}

In particular, for $a_k>a_{k+1}$,
the perturbed GUE corners process
coming from the random matrix
$H=t^{1/2}\cdot G+t\cdot \mathrm{diag}(\mathbf{a})$
(cf. \Cref{sub:matrix_model}),
after the application of 
$\mathscr{S}^{a_k-a_{k+1}}_k$,
turns into the 
corners process
for the random matrix
$\mathscr{T}_k H \mathscr{T}_k \stackrel{d}{=} t^{1/2}\cdot G+t\cdot \mathrm{diag}(\tau_k\mathbf{a})$:
\begin{equation}
	\label{eq:tau_k_action_in_term_of_RM}
	H\xdashrightarrow{\mathscr{S}^{a_k-a_{k+1}}_k} \mathscr{T}_k H\mathscr{T}_k,
\end{equation}
where $\mathscr{T}_k$ is the permutation matrix
of $\tau_k=(k,k+1)$.
In other words, applying the exponential jumps 
$\mathscr{S}^{a_k-a_{k+1}}_k$ 
on the level of eigenvalues
is
equivalent in distribution
to the change of basis $\mathrm{e}_k\leftrightarrow \mathrm{e}_{k+1}$
in the space corresponding to the random matrix.

\section{Global shift and reflected Brownian motions}
\label{sec:proofs_5}

In this section we consider a 
special case when the perturbation sequence is an arithmetic progression, 
and prove \Cref{thm:main_theorem_shift,thm:main_theorem_Brownian}.

Set
\begin{equation}
	\label{eq:a_Gibbs_measures}
	a_j=-(j-1)\alpha,\qquad  j=1,2,\ldots , 
\end{equation}
where $\alpha>0$ is fixed. 
Denote the corresponding random matrix by 
\begin{equation}
	\label{eq:H_alpha}
	H^\alpha=t^{1/2}\cdot G+t\cdot \mathrm{diag}(0,-\alpha,-2\alpha,\ldots ),
\end{equation}
and 
its corners distribution on 
infinite interlacing arrays by $\mathbf{M}^{\alpha}$.
To $\mathbf{M}^\alpha$
we will apply the sequence of swap operators
$\mathscr{S}^{k\alpha}_k$, 
first with $k=1$, then with $k=2$, and so on.
Denote the resulting Markov operator which 
acts on the infinite interlacing array by 
$\mathbb{S}^\alpha$.

\begin{lemma}
	\label{lemma:big_shift_is_well_defined}
	The Markov transition operator $\mathbb{S}^{\alpha}$ is well-defined.
\end{lemma}
\begin{proof}
	Let $\{\lambda^m_j\}_{1\le j\le m<\infty}$
	be the random interlacing
	array to which we apply $\mathbb{S}^{\alpha}$.
	The resulting random interlacing array 
	$\{\nu^m_j\}_{1\le j\le m<\infty}$
	is defined inductively: for each
	$k$, the 
	$k$-th level
	configuration $\nu^k$ 
	is the result of the action of $\mathscr{S}^{k\alpha}_k$
	on $\lambda^k$ given $\nu^{k-1}$ and $\lambda^{k+1}$.
	For this action, the configuration $\nu^{k-1}$ 
	was defined on the previous step of the induction.
	This implies that $\mathbb{S}^{\alpha}$ is well-defined.
\end{proof}

Acting on $\mathbf{a}$-Gibbs measures with
$\mathbf{a}=(0,-\alpha,-2\alpha,\ldots )$
\eqref{eq:a_Gibbs_measures},
$\mathscr{S}^\alpha_1$ interchanges $0$ with $-\alpha$,
then 
$\mathscr{S}^{2\alpha}_2$ interchanges $0$ (which is now the new $a_2$)
with $-2\alpha$, and so on. 
After infinitely many swaps, the parameter $0$ disappears,
and one expects that the resulting distribution would be 
$\mathbf{a}$-Gibbs with $\mathbf{a}=(-\alpha,-2\alpha,-3\alpha,\ldots )$.
For the special choice of the 
perturbed GUE corners process \eqref{eq:H_alpha}, 
the action of $\mathbb{S}^\alpha$ is, moreover, equivalent in distribution
to a global shift. We establish the following result:

\begin{theorem}[\Cref{thm:main_theorem_shift} in Introduction]
	\label{thm:shift_for_alpha_GUE}
	The action of $\mathbb{S}^{\alpha}$ on $\mathbf{M}^\alpha$ is 
	equivalent in distribution to a deterministic shift of the 
	whole infinite interlacing array to the left by $\alpha t$.
	In terms of random matrices, we have
	\begin{equation}
		\label{eq:big_shift_action_in_term_of_RM}
		H^\alpha\xdashrightarrow{\mathbb{S}^\alpha}
		H^\alpha-\alpha t\,\mathbf{I},
	\end{equation}
	where $\mathbf{I}$ is the infinite identity matrix.
\end{theorem}
\begin{proof}
	Informally, one can think that
	\eqref{eq:big_shift_action_in_term_of_RM}
	follows by a sequential application
	of the change of basis 
	\eqref{eq:tau_k_action_in_term_of_RM}
	under a single-level action $\mathscr{S}^{k\alpha}_k$.
	The shift by $\alpha t$ is precisely the difference
	between $t\cdot\mathrm{diag}(\mathbf{a})$ before and 
	after the modification of $\mathbf{a}$.
	We will now prove
	this claim more formally, using \Cref{thm:action_of_level_swap_on_Gibbs}
	on how Gibbs measures change under swap operators.
	
	Take the 
	harmonic functions
	$\varphi_N=\varphi_N^{\mathrm{pertGUE}(\mathbf{a};t)}$
	as in \eqref{eq:pertGUE} with $\mathbf{a}=(0,-\alpha,-2\alpha,\ldots )$.
	The action of each $\mathscr{S}^{k\alpha}_k$
	changes only the $k$-th function $\varphi_k$ as in 
	\eqref{eq:modification_under_harmonic_functions}
	and leaves all other functions intact.
	Therefore, the action of the whole $\mathbb{S}^\alpha$
	replaces $\{\varphi_k\}$ by the family
	\begin{equation}
		\label{eq:global_shift_proof}
		\varphi'_k(\lambda^k)=
		\int_{\lambda^{k+1}\colon \lambda^{k+1}\succ \lambda^k}
		\varphi_{k+1}(\lambda^{k+1})\,
		d\lambda^{k+1}.
	\end{equation}
	Here we took $a_k=0$ because this is precisely the perturbation
	parameter that is being swapped with $a_{k+1}=-k\alpha$
	under the action of $\mathscr{S}^{k\alpha}_k$.
	The integral in the right-hand side of \eqref{eq:global_shift_proof}
	can be computed using 
	\eqref{eq:pertGUE_Gibbs_property} (with $a_N=0$ in that formula), 
	and we obtain
	\begin{equation*}
		\varphi_k'(\lambda^k)
		=
		\mathrm{const}\times
		\mathsf{V}(\lambda^k_1,\ldots,\lambda^k_k )
		\prod_{i=1}^{k}e^{-(\lambda^k_i)^2/(2t)}
		\prod_{j=1}^{k+1}e^{-t( (j-1)\alpha)^2/2}
		=
		C_0 \varphi_k(\lambda^k)\,e^{-tk^2\alpha^2/2}.
	\end{equation*}
	Here both $\mathrm{const}$ and $C_0$
	are some constants which are independent of $\alpha$.
	Sequentially applying \Cref{thm:action_of_level_swap_on_Gibbs},
	we see that the new
	the harmonic functions 
	$\{\varphi_k'\}$ satisfy Gibbs property
	with the sequence $\mathbf{a}=(-\alpha,-2\alpha,-3\alpha,\ldots )$,
	and hence (by \Cref{prop:uniquely}) correspond to a 
	Gibbs measure with shifted parameters. Let us now identify this 
	particular Gibbs measure.

	The modified density
	of $\lambda^k$ after the application of
	$\mathbb{S}^\alpha$ reads, by \eqref{eq:a_harmonic_defn},
	\begin{align*}
		\mathsf{Density}'(\lambda^k)
		&=
		\frac{\det[ \exp\{(-i \alpha)\lambda_j^k\} ]_{i,j=1}^k}
		{\mathsf{V}(-\alpha,-2\alpha,\ldots,-k\alpha )}
		\,\varphi_k'(\lambda^k)
		\\&=
		C_0\,
		\frac{\det[ \exp\{-i \alpha\lambda_j^k\} ]_{i,j=1}^k}
		{\mathsf{V}(-\alpha,-2\alpha,\ldots,-k\alpha )}
		\,
		\varphi_k(\lambda^k)\,e^{-tk^2\alpha^2/2}
		\\&=
		C_0\,
		\frac{\det[ \exp\{-i \alpha\lambda_j^k\} ]_{i,j=1}^k}
		{\mathsf{V}(-\alpha,-2\alpha,\ldots,-k\alpha )}
		\frac
		{\mathsf{V}(0,-\alpha,-2\alpha,\ldots,-(k-1)\alpha )}
		{\det[ \exp\{-(i-1) \alpha\lambda_j^k\} ]_{i,j=1}^k}
		\,
		\mathsf{Density}(\lambda^k)\,
		e^{-tk^2\alpha^2/2}
		\\&=
		C_0\,e^{-\alpha|\lambda^k|-tk^2\alpha^2/2}
		\mathsf{Density}(\lambda^k),
	\end{align*}
	where
	$\mathsf{Density}(\cdot)$
	is the original density before applying $\mathbb{S}^{\alpha}$.
	In the last step, the two Vandermondes are equal by their translation invariance, 
	and the ratio of the determinants is $e^{-\alpha|\lambda^k|}$
	(indeed, factor out $e^{-\lambda^k_j}$ from each $j$-th column of the determinant in the 
	numerator).
	Now, using \eqref{eq:lambda_N_density} we have
	\begin{equation*}
		\begin{split}
			&
			C_0\,e^{-\alpha|\lambda^k|-tk^2\alpha^2/2}
			\,\mathsf{Density}(\lambda^k)
			\\&\hspace{10pt}=
			C_0\mathrm{const}\times
			e^{-\alpha|\lambda^k|-tk^2\alpha^2/2}
			\det\left[ \exp\left\{ 
					-\frac{(\lambda_i^{k}+t (j-1)\alpha)^2}{2t}
			\right\} \right]_{i,j=1}^{k}
			\frac{\mathsf{V}(\lambda_1^{k},\ldots,\lambda^{k}_k )}{\mathsf{V}(0,-\alpha,
			\ldots,-(k-1)\alpha )}.
		\end{split}
	\end{equation*}
	Here $\mathrm{const}$
	is the normalizing constant in \eqref{eq:lambda_N_density}
	which is independent of $\alpha$.
	Observe that in the exponents inside the determinant we have
	\begin{equation*}
		-\frac1{2t}(\lambda^k_i+t(j-1)\alpha)^2=
		-\frac1{2t}(\lambda^k_i+\alpha t+t(j-1)\alpha)^2
		+\alpha \lambda^k_i+\frac{t \alpha^2}{2}(2j-1).
	\end{equation*}
	Factoring out the last two
	terms from each $j$-th column, we get a factor 
	which precisely cancels with 
	$e^{-\alpha|\lambda^k|-tk^2\alpha^2/2}$.
	Therefore, we see that 
	\begin{equation*}
		\mathsf{Density}'(\lambda^k)=C_0
		\mathsf{Density}(\lambda^k+\alpha t).
	\end{equation*}
	Normalizing, this implies that $C_0=1$. 
	Thus, we see that applying $\mathbb{S}^{\alpha}$ 
	is indeed equivalent to the global shift by $\alpha t$ to the left,
	as desired.
\end{proof}

We can now establish the 
shifting property for the reflected Brownian motions:

\begin{proof}[Proof of \Cref{thm:main_theorem_Brownian}]
	Fix $t$, and use the identification 
	$\{\lambda^k_j \}\stackrel{d}{=}\{B^k_j(t)\}$
	of the
	GUE corners distribution with that of the reflected Brownian motions
	from \Cref{prop:connection_to_reflected_BM}
	Denote $X_k(t):=B^k_k(t)$, then these are exactly the reflected Brownian motions
	from \Cref{thm:main_theorem_Brownian}. 
	Observe that the action of the operator $\mathscr{S}^{k\alpha}_k$
	\eqref{eq:swap_operator_at_a_single_level} on these $\lambda^k_k\stackrel{d}{=}X_k$
	depends only on $\lambda^k_k$ and $\lambda^{k+1}_{k+1}$
	and is the same as the Markov operator \eqref{eq:BM_jumps} 
	in \Cref{thm:main_theorem_Brownian}.
	Combining this observation with 
	the shifting property from \Cref{thm:shift_for_alpha_GUE}
	we obtain the desired claim.
\end{proof}

\bibliographystyle{alpha}
\bibliography{bib}

\end{document}